\documentclass[12pt]{amsart}
\usepackage[all]{xy} % for complicated commutative diagrams
\usepackage{amsmath}
\usepackage{amsfonts}
\usepackage{amssymb}

\DeclareFontEncoding{OT2}{}{} % to enable usage of cyrillic fonts
\newcommand{\textcyr}[1]{%
 {\fontencoding{OT2}\fontfamily{wncyr}\fontseries{m}\fontshape{n}
 \selectfont #1}}
\newcommand{\Sha}{{\mbox{\textcyr{Sh}}}}
\newcommand{\Tha}{{\mbox{\textcyr{Q}}}}

\def\bp{{\mathbb P}}
\def\bz{{\mathbb Z}\,}
\def\bh{{\mathbb H}}
\def\bq{{\mathbb Q}}
\def\bn{{\mathbb N}}

\def\be{\kern -.1em}
\def\lbe{\kern -.05em}
\def\s{\mathcal }
\def\ra{\rightarrow}
\def\e{\kern 0.08em}
\def\le{\kern 0.04em}
\def\ng{\kern -0.04em}
\def\spec{{\rm{Spec}}\,}

\def\krn{{\rm{Ker}}\, }
\def\img{{\rm{Im}}\,}
\def\cok{{\rm{Coker}}\,}

\newtheorem{lemma}{Lemma}[section]
\newtheorem{theorem}[lemma]{Theorem}
\newtheorem{corollary}[lemma]{Corollary}
\newtheorem{proposition}[lemma]{Proposition}
\theoremstyle{definition}

\theoremstyle{remark}
\newtheorem{remark}[lemma]{Remark}

\begin{document}

\title[The Generalized Cassels-Tate dual exact sequence]{The
generalized
Cassels-Tate dual exact sequence for 1-motives}


\author[Gonz\'alez-Avil\'es]
{Cristian D. Gonz\'alez-Avil\'es}
\address{Departamento de Matem\'aticas, Universidad de La Serena,
Chile} \email{cgonzalez@userena.cl}

\author[Tan]{Ki-Seng Tan}
\address{Department of Mathematics\\
National Taiwan University\\
Taipei 10764, Taiwan} \email{tan@math.ntu.edu.tw}

\thanks{The first author is partially supported by Fondecyt
grant 1080025. The second author was supported in part by the
National Science Council of Taiwan, 97-2115-M-002-006-MY2}


\begin{abstract} We establish a generalized Cassels-Tate dual
exact sequence for 1-motives over global fields. We thereby extend
the main theorem of [4] from abelian varieties to arbitrary
1-motives.
\end{abstract}

\maketitle

\section{Introduction}

Let $K$ be a global field and let $M=(Y\ra G)$ be a (Deligne)
1-motive over $K$, where $Y$ is \'etale-locally isomorphic to
$\bz^{\!r}$ for some $r\geq 0$ and $G$ is a semiabelian variety over
$K$. Let $M^{*}$ be the 1-motive dual to $M$. If $B$ is a
topological abelian group, $B^{\wedge}$ will denote the completion
of $B$ with respect to the family of open subgroups of finite index.
Let $\Sha^{\e 1}(M)$ (resp. $\Sha^{\e 1}_{\omega}(M)$) denote the
subgroup of $\bh^{\e 1}(K,M)$ of all classes which are locally
trivial at all (resp. all but finitely many) primes of $K$. There
exists a canonical exact sequence of discrete torsion groups
$$
0\ra\Sha^{\e 1}(M)\ra\Sha^{\e 1}_{\omega}(M)\ra \bigoplus_{\text{all
$v$}}\bh^{\e 1}(K_{v},M)\ra\Tha^{1}(M)\ra 0,
$$
where we have written $\!\!\!\Tha^{1}(M)$ for the cokernel of the
middle map. By the local duality theorem for 1-motives [7, Theorem
2.3 and Proposition 2.9], the Pontryagin dual of the above exact
sequence is an exact sequence
$$
0\ra\Tha^{1}(M)^{D}\ra\prod_{\text{all $v$}}\bh^{\e
0}(K_{v},M^{*})^{\wedge}\ra\Sha^{\e 1}_{\omega}(M)^{D} \ra\Sha^{\e
1}(M)^{D}\ra 0,
$$
where each group $\bh^{\e 0}(K_{v},M^{*})$ is endowed with the
topology defined in [7, p.99]. A fundamental problem is to describe
$\!\!\Tha^{1}(M)^{D}$. This problem was first addressed in the case
of elliptic curves $E$ over number fields $K$ (i.e., $Y=0$ and $G=E$
above), by J.W.S.Cassels (see [2, Theorem 7.1] and [3, Appendix 2]).
Cassels showed that $\!\!\Tha^{1}(E^{*})^{D}$ is canonically
isomorphic to the pro-Selmer group $T\e{\rm{Sel}}(E)$ of $E$. This
result was extended to abelian varieties $A$ over number fields $K$
by J.Tate, under the assumption that $\!\!\Sha^{\e 1}(A)$ is finite
(unpublished). In this case $T\e{\rm{Sel}}(A)$ is isomorphic to
$H^{\e 0}(K,A)^{\wedge}$ and $\bh^{\e 0}(K_{v},M)^{\wedge}=H^{\e
0}(K_{v},A)^{\wedge}=H^{\e 0}(K_{v},A)$ for any $v$ since $H^{\e
0}(K_{v},A)$ is profinite. Further, $\Sha^{\e
1}_{\omega}(A^{*})=H^{\e 1}(K,A^{*})$ and $\Sha^{\e
1}(A^{*})^{D}=\Sha^{\e 1}(A)$. The exact sequence obtained by Tate,
now known as the {\it Cassels-Tate dual exact sequence}, is
\begin{equation}
0\ra H^{\e 0}(K,A)^{\wedge}\ra\prod_{\text{all $v$}}H^{\e
0}(K_{v},A)\ra H^{\e 1}(K,A^{*})^{D}\ra\Sha^{\e 1}(A)\ra 0.
\end{equation}
Further, the image of $H^{\e 0}(K,A)^{\wedge}$ is isomorphic to the
closure $\overline{H^{\e 0}(K,A)}$ of the diagonal image of $H^{\e
0}(K,A)$ in $\prod_{\,\text{all $v$}}H^{\e 0}(K_{v},A)$. See [11,
Remark I.6.14(b), p.102]. The preceding exact sequence was recently
extended to arbitrary 1-motives over number fields by D.Harari and
T.Szamuely [7, Theorem 1.2], again under the assumption that
$\Sha^{\e 1}(M)$ is finite. These authors established the exactness
of the sequence
$$
0\ra\overline{\bh^{\e 0}(K,M)}\ra\prod_{\text{all $v$}}\bh^{\e
0}(K_{v},M)\ra \Sha^{\e 1}_{\omega}(M^{*})^{D}\ra\Sha^{\e 1}(M)\ra
0,
$$
where the middle map is induced by the local pairings of [7, \S2].
This natural analogue of (1) was used in [op.cit., \S6] to study
weak approximation on semiabelian varieties over number fields.
However, it does not provide a description of $\!\!\Tha^{1}(M)^{D}$
when $\!\!\Sha^{\e 1}(M)$ is finite. Our objective in this paper is
to describe $\!\!\Tha^{1}(M)^{D}$ for any $K$ independently of the
finiteness assumption on $\!\!\Sha^{\e 1}(M)$. In order to state our
main result, let
$$
\text{Sel}(M)_{n}=\krn\!\left[\e H^{\e 1}(K, T_{\e\bz\be/n}(M))\ra
\displaystyle\prod_{\text{all $v$}}{\bh}^{\e 1}(K_{v},M)_{n}\right]
$$
be the $n$-th Selmer group of $M$, where $n$ is any positive integer
and $T_{\e\bz\be/n}(M)$ is the $n$-adic realization of $M$. Let
$T\text{Sel}(M)=\varprojlim_{n}\text{Sel}(M)_{n}$ be the pro-Selmer
group of $M$. Our main theorem is the following result.

\begin{theorem} {\rm{(The generalized Cassels-Tate dual exact
sequence for 1-motives)}}. Let $M$ be a 1-motive over a global field
$K$. Then there exists a canonical exact sequence of profinite
groups
$$\begin{array}{rcl}
0\ra\Sha^{\e 2}(M^{*})^{D}&\ra & T\e{\rm{Sel}}(M)^{\wedge}\ra
\displaystyle\prod_{{\rm{all}}\,v}\bh^{\e
0}(K_{v},M)^{\wedge}\\
&\ra &\Sha^{\e 1}_{\omega}(M^{*})^{D}\ra \Sha^{\e 1}(M^{*})^{D} \ra
0.
\end{array}
$$
\end{theorem}

If $M=(0\ra A)$ is an abelian variety, then $\Sha^{\e
2}(M^{*})^{D}=0$ and we recover the main theorem of [4].
Applications of Theorem 1.1 will be given in [6].

\section*{Acknowledgements}

We thank C.U.Jensen, P.Jossen and J.S.Milne for helpful comments.

\section{Preliminaries}

Let $K$ be a global field, i.e. $K$ is a finite extension of $\bq$
(the ``number field case") or is finitely generated and of
transcendence degree 1 over a finite field of constants $k$ (the
``function field case"). For any prime $v$ of $K$, $K_{v}$ will
denote the completion of $K$ at $v$ and ${\s O}_{v}$ will denote the
corresponding ring of integers. Thus ${\s O}_{v}$ is a complete
discrete valuation ring. Further, $X$ will denote either the
spectrum of the ring of integers of $K$ (in the number field case)
or the unique smooth complete curve over $k$ with function field $K$
(in the function field case).

\smallskip

All cohomology groups below are flat (fppf) cohomology groups.

\smallskip

For any topological abelian group $B$, we set
$B^{D}=\text{Hom}_{\text{cont.}}(B,\bq/\bz)$ and endow it with the
compact-open topology, where $\bq/\bz$ carries the discrete
topology. If $n$ is any positive integer, $B/n$ will denote $B/nB$
with the quotient topology. Let $B_{\wedge}=\varprojlim_{\e
n\in\bn}B/n$ with the inverse limit topology. Further, define
$B^{\wedge}=\varprojlim_{\e U\in \s U}B/U$, where $\s U$ denotes the
family of open subgroups of finite index in $B$. If
$B_{\sim}:=\varprojlim_{\e n\in\bn}B\big/\e\overline{nB}$, where
$\overline{nB}$ denotes the closure of $nB$ in $B$, then there
exists a canonical isomorphism $(B_{\sim})^{\wedge}=B^{\wedge}$.
Consequently, there exists a canonical map $B_{\wedge}\ra
B^{\wedge}$. If $B$ is discrete (or compact), then
$B_{\sim}=B_{\wedge}$ and therefore
$(B_{\wedge})^{\wedge}=B^{\wedge}$. We also note that $B^{\wedge}=B$
if $B$ is profinite (see, e.g., [14, Theorem 2.1.3, p.22]). For any
positive integer $n$, $B_{n}$ will denote the $n$-torsion subgroup
of $B$ and $T\e B=\varprojlim_{\e n\in\bn}B_{n}$ is the total Tate
module of $B$. Note that $TB=0$ if $B$ is finite.

\smallskip

Let $M=(Y\ra G)$ be a Deligne 1-motive over $K$, where $Y$ is
\'etale-locally isomorphic to $\bz^{\! r}$ for some $r$ and $G$ is a
semiabelian variety (for basic information on 1-motives over global
fields, see [7, \S1] or [5, \S3]. Let $n$ be a positive integer. The
{\it $n$-adic realization of $M$} is a finite and flat $K$-group
scheme $T_{\bz/n}(M)$ which fits into an exact sequence
$$
0\ra G_{n}\ra T_{\bz/n}(M)\ra Y/n\ra 0.
$$
There exists a perfect pairing
$$
T_{\bz/n}(M)\times T_{\bz/n}(M^{*})\ra \mu_{\e n},
$$
where $\mu_{\e n}$ is the sheaf of $n$-th roots of unity. Further,
given positive integers $n$ and $m$ with $n\!\mid\! m$, there exist
canonical maps $T_{\bz/n}(M)\ra T_{\bz/m}(M)$ and $T_{\bz/m}(M)\ra
T_{\bz/n}(M)$. Let $T(M)_{\text{tors}}=\varinjlim T_{\bz/n}(M)$.
Further, for any $i\geq 0$, define
$$
H^{\e i}(K,T(M))=\displaystyle\varprojlim_{n} H^{\e
i}(K,T_{\bz/n}(M)).
$$
If $v$ is archimedean and $i\geq -1$, $\bh^{\e i}(K_{v}, M)$ will
denote the (finite, 2-torsion) {\it reduced} (Tate) hypercohomology
groups of $M_{K_{v}}$ defined in [7, p.103]. All groups $\bh^{\e
i}(K_{v},M)$ will be given the discrete topology, except for
$\bh^{\e 0}(K_{v},M)$ for non-archimedean $v$. The latter group will
be given the topology defined in [7, p.99]. Thus there exists an
exact sequence $0\ra I\ra\bh^{\e 0}(K_{v},M)\ra F\ra 0$, where $F$
is finite and $I$ is an open subgroup of $\bh^{\e 0}(K_{v},M)$ which
is isomorphic to $G(K_{v})/L$ for some finitely generated subgroup
$L$ of $G(K_{v})$. If $n$ is a positive integer, $G(K_{v})/n$ is
profinite (see [5, beginning of \S5]). Thus the exactness of
$$
L/n\ra G(K_{v})/n\ra I/n\ra 0
$$
shows that $I/n$ is profinite as well. Now the exactness of
$$
F_{n}\ra I/n\ra\bh^{\e 0}(K_{v},M)/n\ra F/n\ra 0
$$
shows that $\bh^{\e 0}(K_{v},M)/n$ is profinite (see [14,
Proposition 2.2.1(e), p.28]). The latter also holds if $v$ is
archimedean. We conclude that $\bh^{\e 0}(K_{v},M)_{\wedge}$ is
profinite for every $v$ (see [14, Proposition 2.2.1(d), p.28]).

\medskip

All groups $\bh^{\e i}(K,M)$ will be endowed with the discrete
topology.

\begin{lemma} $\bh^{\e 0}(K,M)_{\wedge}$ is Hausdorff, locally
compact and $\sigma$-compact.
\end{lemma}
\begin{proof} This follows from the fact that
$\bh^{\e 0}(K,M)_{\wedge}$ is topologically isomorphic to a
countable direct limit of compact groups. Indeed, there exists a
canonical isomorphism
$$
\bh^{\e 0}(K,M)_{\wedge}=\varinjlim_{(U,\,\s M\e)\e\in\e\s F}\bh^{\e
0}(U,\s M)_{\wedge},
$$
where $\s F$ is the set of all pairs $(U,\s M\e)$ such that $U$ is a
nonempty open affine subscheme of $X$ and $\s M$ is a 1-motive over
$U$ which extends $M$ (cf. [5, proof of Lemma 2.3]). By [7, Lemma
3.2(3), p.107], each $\bh^{\e 0}(U,\s M)_{\wedge}$ is profinite.
Further, since the complement of $U$ in $X$ is a finite set of
primes of $K$ and $K$ has only countably many primes, $\s F$ is
countable.
\end{proof}

For each $i\geq 0$, let ${\bp}^{\e i}\!\left(M\be\right)$ be the
restricted direct product over all primes of $K$ of the groups
$\bh^{\e i}(K_{v},M)$ with respect to the subgroups
$$
\bh^{\e i}_{\e\text{nr}}(K_{v},M)=\img\be\!\left[\bh^{\e i}(\s
O_{v},\s M)\ra\bh^{\e i}(K_{v},M)\right]
$$
for $v\in U$, where $U$ is any nonempty open subscheme of $X$ such
that $M$ extends to a 1-motive $\s M$ over $U$. The groups $P^{\e
i}\!\left(F\be\right)$ are defined similarly for any abelian fppf
sheaf $F$ on $\spec K$. By [7, Lemma 5.3]\footnote{ This result and
its proof remain valid in the function field case, using the fact
that $H^{\e 1}_{v}({\s O}_{v},T_{\bz\be/\be p^{m}}(\s M))=0$ for any
$m$ by [13, beginning of \S 7, p.349].}, ${\bp}^{\e
0}\!\left(M\be\right)_{\wedge}$ is the restricted direct product of
the groups $\bh^{\e 0}(K_{v},M)_{\wedge}$ with respect to the
subgroups $\bh^{\e 0}(\s O_{v},\s M)_{\wedge}$. It is therefore
Hausdorff, locally compact and $\sigma$-compact (see [10, 6.16(c),
p.57]). Further, by [7, Theorems 2.3 and 2.10], the dual of
${\bp}^{\e 0}\!\left(M\be\right)_{\wedge}$ is ${\bp}^{\e
1}\!\left(M^{*}\be\right)$, whence the dual of the profinite group
${\bp}^{\e 0}\!\left(M\be\right)^{\wedge}$ is the discrete torsion
group ${\bp}^{\e 1}\!\left(M^{*}\be\right)_{\text{tors}}$.

Recall that a morphism $f\colon A\ra B$ of topological groups is
said to be {\it strict} if the induced map $A/\e\krn f\ra\img f$ is
an isomorphism of topological groups. Equivalently, $f$ is strict if
it is open onto its image [1, \S III.2.8, Proposition 24(b), p.236].
We will need the following

\begin{lemma} Let $A\overset{f}\longrightarrow B\overset{g}
\longrightarrow C$ be an exact sequence of abelian topological
groups and strict morphisms. If $C\ra C^{\wedge}$ is injective, then
$A^{\wedge}\overset{\widehat{f}}\longrightarrow
B^{\wedge}\overset{\widehat{g}} \longrightarrow C^{\wedge}$ is also
exact.
\end{lemma}
\begin{proof} The map $A\ra\img f$ induced by $f$ is an open
surjection, so $A^{\wedge}\ra(\img f)^{\wedge}$ is surjective as
well. Further, since $B\ra\img g$ is an open surjection, the
sequence $(\img f)^{\wedge}\ra B^{\wedge}\ra(\img g)^{\wedge}\ra 0$
is exact [7, Appendix]. Finally, since $C$ injects into
$C^{\wedge}$, $(\img g)^{\wedge}$ is the closure of $\img g$ in
$C^{\wedge}$, whence $(\img g)^{\wedge}\ra C^{\wedge}$ is injective.
\end{proof}

\section{The Poitou-Tate exact sequence for 1-motives
over function fields}

For any positive integer $n$, there exists a canonical exact
commutative diagram
\begin{equation}
\xymatrix{0\ar[r]&\bh^{\e 0}(K,M)/n\ar[d]\ar[r] & H^{\e 1}(K,
T_{\e\bz\be/n}(M))\ar[r]\ar[d]& \bh^{\e
1}(K,M)_{n}\ar[d]\ar[r]&0\\
0\ar[r]&{\bp}^{\e 0}\!\left(M\be\right)/n\ar[r] & P^{\e
1}(T_{\e\bz\be/n}(M))\ar[r]& {\bp}^{\e
1}\!\left(M\be\right)_{n}\ar[r]&0,\\
}
\end{equation}
whose vertical maps are induced by the canonical morphisms $\spec
K_{v}\ra\spec K$. For the exactness of the rows, see [7, p.109].
Now, for any $i\geq -1$, set
$$
\Sha^{\e i}\!\left(M\be\right)=\krn\!\left[\,\bh^{\e i}(K,M)\ra
{\bp}^{\e i}\!\left(M\be\right)\,\right].
$$
Further, define
$$
\text{Sel}(M)_{n}=\krn\!\left[\e H^{\e 1}(K, T_{\e\bz\be/n}(M))\ra
{\bp}^{\e 1}\!\left(M\be\right)_{n}\right],
$$
where the map involved is the composite
$$
H^{\e 1}(K,T_{\e\bz\be/n}(M))\ra P^{\e 1}(\e T_{\e\bz\be/n}(M))\ra
{\bp}^{\e 1}\!\left(M\be\right)_{n}.
$$
Now set
$$\begin{array}{rcl}
T\text{Sel}(M)&=&\displaystyle\varprojlim_{n}\text{Sel}(M)_{n}\\
P^{\e 1}(T(M))&=&\displaystyle\varprojlim_{n}P^{\e
1}(T_{\e\bz\be/n}(M)).
\end{array}
$$
Since $(\e\bh^{\e 0}(K,M)/n)$ and $(\e{\bp}^{\e
0}\!\left(M\be\right)/n)$ are inverse systems with surjective
transition maps, the inverse limit of (2) is an exact commutative
diagram
\begin{equation}
\xymatrix{0\ar[r]&\bh^{\e 0}(K,M)_{\wedge}\ar[d]\ar[r] & H^{\e 1}(K,
T(M))\ar[r]\ar[d]& T\,\bh^{\e
1}(K,M)\ar[d]\ar[r]&0\\
0\ar[r]&{\bp}^{\e 0}\!\left(M\be\right)_{\wedge}\ar[r] & P^{\e
1}(T(M))\ar[r]& T\,{\bp}^{\e
1}\!\left(M\be\right)\ar[r]&0\\
}
\end{equation}
(see, e.g., [9, Example 9.1.1, p.192]). The above diagram yields an
exact sequence
\begin{equation}
0\ra\bh^{\e 0}(K,M)_{\wedge}\ra T\e\text{Sel}(M)\ra T\!\be\Sha^{\e
1}(M)\ra 0.
\end{equation}
Thus, if $\Sha^{\e 1}(M)$ is finite, then $T\e\text{Sel}(M)$ is
canonically isomorphic to $\bh^{\e 0}(K,M)_{\wedge}$. In particular,
$T\e\text{Sel}(M)^{\wedge}=(\bh^{\e 0}(K,M)_{\wedge})^{\wedge}=
\bh^{\e 0}(K,M)^{\wedge}$.

\begin{lemma} $T\e{\rm{Sel}}(M)$ is locally compact and
$\sigma$-compact.
\end{lemma}
\begin{proof} By [7, Lemma 3.2(2)] and [5, Lemma 6.5],
$\Sha^{\e 1}\!\left(M\be\right)_{n}$ is finite for any $n$. Thus
$T\!\Sha^{\e 1}(M)$ is profinite and the lemma follows from (4),
Lemma 2.1 and [15, Theorem 6.10(c), p.57].
\end{proof}

\begin{lemma} The canonical map $H^{\e
1}(K, T(M))\twoheadrightarrow T\,\bh^{\e 1}(K,M)$ appearing in
diagram (3) induces an isomorphism
$$
H^{\e 1}(K, T(M))/\e T\e{\rm{Sel}}(M)\simeq T\,\bh^{\e 1}(K,M)/\e
T\!\be\Sha^{\e 1}(M).
$$
\end{lemma}
\begin{proof} This is immediate from diagram (3) and the definitions
of $\!\!\Sha^{\e 1}(M)$ and $T\e{\rm{Sel}}(M)$.
\end{proof}

By definition of $T\text{Sel}(M)$, diagram (3) induces a canonical
map
$$
\theta_{0}\colon T\text{Sel}(M)\ra {\bp}^{\e
0}\!\left(M\be\right)_{\wedge}.
$$

\begin{proposition} There exists a perfect pairing
$$
\krn\theta_{\e 0}\times\!\!\Sha^{\e 2}(M^{*})\ra\bq/\bz,
$$
where the first group is profinite and the second is discrete and
torsion.
\end{proposition}
\begin{proof} (Cf. [7, proof of Proposition 5.1, p.119]) There
exists a canonical exact commutative diagram
\begin{equation}
\xymatrix{0\ar[r]&T\e{\rm{Sel}}(M)\ar[d]^{\theta_{0}}\ar[r] & H^{\e
1}(K, T(M))\ar@{->>}[r]\ar[d]^{\theta}& T\,\bh^{\e 1}(K,M)/\e
T\!\Sha^{\e 1}\!\left(M\be\right)\ar@{^{(}->}[d]
\\
0\ar[r]&{\bp}^{\e 0}\!\left(M\be\right)_{\wedge}\ar[r] & P^{\e
1}(T(M))\ar@{->>}[r]& T\,{\bp}^{\e
1}\!\left(M\be\right).\\
}
\end{equation}
The top row is exact by Lemma 3.2. Clearly,
$\krn(\theta_{0})=\krn(\theta)$. Now, by Poitou-Tate duality for
finite modules ([13, Theorem I.4.10, p.70] and [5, Theorem 4.9]),
for each $n$ there exists a perfect pairing of finite groups
$$
\Sha^{\e 1}(T_{\bz/n}(M))\times\Sha^{\e
2}(T_{\bz/n}(M^{*}))\ra\bq/\bz.
$$
We conclude that $\krn(\theta)=\varprojlim_{n}\!\!\Sha^{\e
1}(T_{\bz/n}(M))$ is canonically dual to $\Sha^{\e
2}\left(T\be\left(M^{*}\right)_{\text{tors}}\right):=
\varinjlim_{\,n}\!\Sha^{\e 2}(T_{\bz/n}(M^{*}))$. But [5, proof of
Lemma 5.8(a)] shows that $\Sha^{\e
2}(T(M^{*})_{\text{tors}})=\Sha^{\e 2}(M^{*})$, which completes the
proof.
\end{proof}

\begin{remark} In the number field case, $\!\!\!\Sha^{\e 2}(M^{*})$ is
known to be finite [12]. Further [op.cit., proof of Theorem 9.4],
the finite group $\krn(\theta_{0})=\varprojlim_{n}\!\!\Sha^{\e
1}(T_{\bz/n}(M))$ is canonically isomorphic to
$$
\krn\!\be\left[\,\bh^{0}(K,M)\ra\prod_{\,\text{all $v$}}
\bh^{0}(K_{v},M)_{\wedge}\e\right],
$$
which conjecturally is the same as $\!\!\Sha^{\e 0}(M)$.
\end{remark}

\begin{lemma} $\theta_{\e 0}$ is a strict morphism.
\end{lemma}
\begin{proof} By Lemma 3.1, [10, Theorem 5.29, p.42] and
[15, Theorem 4.8, p.45], it suffices to check that $\img\theta_{0}$
is a closed subgroup of the locally compact Hausdorff group
${\bp}^{\e 0}\!\left(M\be\right)_{\wedge}$. The image of the map
$\theta$ in diagram (5) can be identified with the kernel of the map
$$
P^{1}(T(M))\ra H^{\e 1}(K,T(M^{*})_{\text{tors}})^{D}
$$
coming from the Poitou-Tate exact sequence for finite modules ([13,
Theorem I.4.10, p.70] and [5, Theorem 4.12])\footnote{ This uses the
fact that $\varprojlim_{\,n}^{(1)}\!\!\!\Sha^{\e
1}(T_{\bz/n}(M))=0$, which holds since each $\!\!\Sha^{\e
1}(T_{\bz/n}(M))$ is finite. See [11, Proposition 2.3, p.14].}. Now
diagram (5) shows that $\img\theta_{0}$ can be identified with the
kernel of the continuous composite map
$$
{\bp}^{\e 0}\!\left(M\be\right)_{\wedge}\ra P^{1}(T(M))\ra H^{\e
1}(K,T(M^{*})_{\text{tors}})^{D}.
$$
Thus $\img\theta_{0}$ is indeed closed in ${\bp}^{\e
0}\!\left(M\be\right)_{\wedge}$.
\end{proof}

There exists a natural commutative diagram
\begin{equation}
\xymatrix{T\e\text{Sel}(M)\ar[d]\ar[r]^{\theta_{0}}&{\bp}^{\e
0}\!\left(M\be\right)_{\wedge}\ar[d]\\
T\e\text{Sel}(M)^{\wedge}\ar[r]^{\beta_{0}}&{\bp}^{\e
0}\!\left(M\be\right)^{\wedge},
\\}
\end{equation}
where $\beta_{0}=\widehat{\theta}_{0}$.

\begin{lemma} The vertical maps in the preceding diagram are
injective.
\end{lemma}
\begin{proof} (Cf. [7, proof of Proposition 5.4, p.119]) Let
$\xi=(\xi_{n})\in T\e\text{Sel}(M)$ be nonzero. Then, for some $n$,
$\xi_{n}\in\text{Sel}(M)_{n}$ is nonzero. Since the canonical map
$\text{Sel}(M)_{n}\ra\text{Sel}(M)_{n}^{\wedge}$ is injective by [7,
Lemma 5.5], we conclude that the image of $\xi_{n}$ in
$\text{Sel}(M)_{n}^{\wedge}$ is nonzero. Consequently, there exists
a subgroup $N$ of $\text{Sel}(M)_{n}$, of finite index, such that
$\xi_{n}\notin N$. It follows that $\xi$ is not contained in the
inverse image of $N$ under the canonical map $T\e\text{Sel}(M)\ra
\text{Sel}(M)_{n}$, which is an open subgroup of finite index in
$T\e\text{Sel}(M)$. We conclude that the image of $\xi$ in
$T\e\text{Sel}(M)^{\wedge}$ is nonzero. This proves the injectivity
of the left-hand vertical map in diagram (6). To prove the
injectivity of the right-hand vertical map, let
$x=(x_{v})\in{\bp}^{\e 0}\!\left(M\be\right)_{\wedge}$ be nonzero.
Then $x\notin n\e {\bp}^{\e 0}\!\left(M\be\right)$ for some $n$,
whence $x_{v}\notin n\e\bh^{\e 0}(K_{v},M)$ for some $v$ (see [7,
Lemma 5.3, p.118]). Thus the image of $x$ under the canonical map
$$
{\bp}^{\e 0}\!\left(M\be\right)_{\wedge}\ra \bh^{\e
0}(K_{v},M)/n=(\bh^{\e 0}(K_{v},M)/n)^{\wedge}
$$
is nonzero, where the equality comes from the fact that $\bh^{\e
0}(K_{v},M)/n$ is profinite. But the preceding map factors through
${\bp}^{\e 0}\!\left(M\be\right)^{\wedge}$, so the image of $x$ in
${\bp}^{\e 0}\!\left(M\be\right)^{\wedge}$ is nonzero.
\end{proof}

\begin{proposition} The map $\krn\theta_{0}\ra\krn\beta_{0}$
induced by diagram (6) is an isomorphism.
\end{proposition}
\begin{proof} The injectivity of the above map is immediate from
Lemma 3.6. Now, by Lemmas 2.2, 3.5 and 3.6, the exact sequence
$$
\krn\theta_{0}\ra
T\e\text{Sel}(M)\overset{\theta_{0}}\longrightarrow {\bp}^{\e
0}\!\left(M\be\right)_{\wedge}
$$
induces an exact sequence
$$
(\krn\theta_{0})^{\wedge}\ra
T\e\text{Sel}(M)^{\wedge}\overset{\beta_{0}}\longrightarrow
{\bp}^{\e 0}\!\left(M\be\right)^{\wedge}.
$$
But $(\krn\theta_{0})^{\wedge}=\krn\theta_{0}$ since
$\krn\theta_{0}$ is profinite by Proposition 3.3, so
$\krn\theta_{0}\ra\krn\beta_{0}$ is indeed surjective.
\end{proof}

For each $v$ and any $n\geq 1$, there exists a canonical pairing
$$
(-,-)_{v}\colon\bh^{\e 0}(K_{v},M)/n\times\bh^{\e
1}(K_{v},M^{*})_{n}\ra\bq/\bz
$$
which vanishes on $\bh^{\e 0}_{\e\text{nr}}(K_{v},M)/n\times \bh^{\e
1}_{\e\text{nr}}(K_{v},M^{*})_{n}$. See [7, p.99 and proof of
Theorem 2.10, p.104]. Let $\gamma_{0,n}^{\e\prime}\colon {\bp}^{\e
0}\!\left(M\be\right)/n\ra(\bh^{\e 1}(K,M^{*})_{n}\be)^{D}$ be
defined as follows. For $x=(x_{v})\in {\bp}^{\e
0}\!\left(M\be\right)/n$ and $\xi\in\bh^{\e 1}(K,M^{*})_{n}$, set
$$
\gamma_{0,n}^{\e\prime}(x)(\xi)=\sum_{\text{all
$v$}}\,(x_{v},\xi\!\be\mid_{K_{v}})_{v},
$$
where $\xi\!\be\mid_{K_{v}}$ is the image of $\xi$ under the
canonical map $\bh^{\e 1}(K,M^{*})_{n}\ra\bh^{\e
1}(K_{v},M^{*})_{n}$ (the sum is actually finite since $x_{v}\in
\bh^{\e 0}_{\e\text{nr}}(K_{v},M)/n$ and
$\xi\!\mid_{K_{v}}\in\bh^{\e 1}_{\e\text{nr}}(K_{v},M)_{n}$ for all
but finitely many primes $v$). Consider the map
$$
\gamma_{0}^{\e\prime}:=\displaystyle\varprojlim_{n}\gamma_{0,n}
^{\e\prime}\colon{\bp}^{\e 0}\!\left(M\be\right)_{\wedge}\ra\bh^{\e
1}(K,M^{*})^{D}.
$$
By [7, p.122], the sequence
\begin{equation}
T\text{Sel}(M)\overset{\theta_{0}}\longrightarrow{\bp}^{\e
0}\!\left(M\be\right)_{\wedge}\overset{\gamma_{0}^{\e\prime}}
\longrightarrow\bh^{\e 1}(K,M^{*})^{D}
\end{equation}
is a complex.
\begin{lemma} The sequence (7) is exact.
\end{lemma}
\begin{proof} As noted in the proof of Lemma 3.5,
$\img\theta_{0}$ is the kernel of the composite map
$$
{\bp}^{\e 0}\!\left(M\be\right)_{\wedge}\ra P^{1}(T(M))\ra H^{\e
1}(K,T(M^{*})_{\text{tors}})^{D}.
$$
Further, there exists a canonical commutative diagram
$$
\xymatrix{{\bp}^{\e
0}\!\left(M\be\right)_{\wedge}\ar[d]^{\gamma_{0}^{\prime}}
\ar[r]&P^{1}(T(M))\ar[d]\\
\bh^{\e 1}(K,M^{*})^{D}\,\ar@{^{(}->}[r]&H^{\e
1}(K,T(M^{*})_{\text{tors}})^{D},
\\}
$$
where the bottom map is the dual of the surjection of discrete
groups $H^{\e 1}(K,T(M^{*})_{\text{tors}})\ra\bh^{\e 1}(K,M^{*})$
(the latter map is the direct limit over $n$ of the surjections
appearing on the top row of diagram (2) for $M^{*}$). We conclude
that $\img\theta_{0}=\krn\gamma_{0}^{\prime}$, as claimed.
\end{proof}

\begin{lemma} $\gamma_{0}^{\e\prime}$ is a strict morphism.
\end{lemma}
\begin{proof} Since ${\bp}^{\e 0}\!\left(M\be\right)_{\wedge}$ is
locally compact and $\sigma$-compact and $\bh^{\e 1}(K,M^{*})^{D}$
is profinite, it suffices to check, by [10, Theorem 5.29, p.42] and
[15, Theorem 4.8, p.45], that $\img\gamma_{0}^{\e\prime}$ is closed
in $\bh^{\e 1}(K,M^{*})^{D}$ (cf. proof of Lemma 3.5). By Lemma 3.8
and diagram (5), $\img\gamma_{0}^{\e\prime}=\cok\theta_{0}$ (with
the quotient topology) injects as a closed subgroup of $\cok\theta$.
On the other hand, the Poitou-Tate exact sequence for finite modules
([13, Theorem I.4.10, p.70] and [5, Theorem 4.12]) shows that
$\cok\theta$ is a closed subgroup of the compact group $H^{\e
1}(K,T(M^{*})_{\text{tors}})^{D}$. It follows that
$\img\gamma_{0}^{\e\prime}$ is a compact (and hence closed) subgroup
of the Hausdorff group $\bh^{\e 1}(K,M^{*})^{D}$.
\end{proof}
Now consider
$$
\gamma_{0}=(\gamma_{0}^{\e\prime})^{\wedge}\colon {\bp}^{\e
0}\!\left(M\be\right)^{\wedge}\ra\left(\bh^{\e
1}(K,M^{*})^{D}\right)^{\wedge}=\bh^{\e 1}(K,M^{*})^{D}.
$$

\begin{proposition} The sequence
$$
T{\rm{Sel}}(M)^{\wedge}\overset{\beta_{0}} \longrightarrow{\bp}^{\e
0}\!\left(M\be\right)^{\wedge}\overset{\gamma_{0}}
\longrightarrow\bh^{\e 1}(K,M^{*})^{D},
$$
is exact.
\end{proposition}
\begin{proof} This follows by applying Lemma 2.2 to the exact
sequence (7) using Lemmas 3.5 and 3.9.
\end{proof}

The following is the main result of this Section. It extends [7,
Theorem 5.6, p.120] to the function field case.

\begin{theorem} Let $K$ be a global function field and let $M$
be a 1-motive over $K$. Assume that $\!\!\Sha^{\e 1}(M)$ is finite.
Then there exists a canonical 12-term exact sequence
\[
\xymatrix{{\bh}^{-1}(K,M)^{\wedge}\,\ar@{^{(}->}[r]^(.45){\gamma_{2}
^{D}}& \,\displaystyle\prod_{{\rm{all}}\,v}{\bh}^{\e
2}(K_{v},M^{*})^{D}\ar[r]^(.55){\beta_{2}^{D}}&\ar[d]{\bh}^{\e
2}(K,M^{*})^{D}\\
\bh^{\e 1}(K,M^{*})^{D}\ar[d]&\ar[l]_{\gamma_{0}}{\bp}^{\e
0}(M)^{\wedge}&\ar[l]_{\beta_{0}}
\bh^{\e 0}(K,M)^{\wedge}\\
\bh^{\e 1}(K,M)\ar[r]^{\beta_{1}}&{\bp}^{\e
1}(M)_{{\rm{tors}}}\ar[r]^(.45){\gamma_{1}}&(\bh^{\e
0}(K,M^{*})^{D})_{{\rm{tors}}}\ar[d]\\
{\bh}^{-1}(K,M^{*})^{D}&\ar@{->>}[l]_{\gamma_{2}}
\displaystyle\bigoplus_{{\rm{all}}\, v}{\bh}^{\e 2}(K_{v},M)&
\ar[l]_(.45){\beta_{2}}{\bh}^{\e 2}(K,M),}
\]
where the maps $\beta_{i}$ are canonical localization maps, the maps
$\gamma_{i}$ are induced by local duality and the unlabeled maps are
defined in the proof.
\end{theorem}
\begin{proof} The exactness of the first line follows as in
[7, p.122], using [5, Theorem 4.12] and noting that [7, Lemma 5.8]
remains valid (with the same proof) in the function field case. The
top right-hand vertical map ${\bh}^{\e 2}(K,M^{*})^{D}\ra \bh^{\e
0}(K,M)^{\wedge}$ is the composite
$$\begin{array}{rcl}
{\bh}^{\e 2}(K,M^{*})^{D}\twoheadrightarrow \Sha^{\e
2}(M^{*})^{D}&\overset{\sim}\longrightarrow&\krn\theta_{\e
0}\overset{\sim}\longrightarrow\krn\beta_{\e
0}\\
&\hookrightarrow&T{\rm{Sel}}(M)^{\wedge}=\bh^{\e 0}(K,M)^{\wedge},
\end{array}
$$
where the isomorphisms come from Propositions 3.3 and 3.7 and the
equality is a consequence of the finiteness hypothesis on
$\!\!\!\Sha^{\e 1}(M)$. The exactness of the second line of the
sequence of the theorem is the content of Proposition 3.10 (again
using the equality $T{\rm{Sel}}(M)^{\wedge}=\bh^{\e
0}(K,M)^{\wedge}$). Since $\gamma_{0}$ is the dual of the natural
map $\bh^{\e 1}(K,M^{*})\ra\bp^{\e 1}(M^{*})_{\text{tors}}$ and
$\Sha^{\e 1}\be(M^{*})^{D}\simeq\Sha^{\e 1}\be(M)$ by [7, Corollary
4.9 and Remark 5.10] and [5, corollary 6.7], we conclude that there
exists an exact sequence
$$
\xymatrix{0\ar[r]&{\bh}^{-1}(K,M)^{\wedge}\,\ar@{^{(}->}[r]^(.45)
{\gamma_{2}
^{D}}&\,\displaystyle\prod_{{\rm{all}}\,v}{\bh}^{\e
2}(K_{v},M^{*})^{D}\ar[r]^(.55){\beta_{2}^{D}}&\ar[d]{\bh}^{\e
2}(K,M^{*})^{D}\\
&\ar@{->>}[d]\bh^{\e 1}(K,M^{*})^{D}&\ar[l]_{\gamma_{0}}{\bp}^{\e
0}\!\left(M\be\right)^{\wedge}&\ar[l]_{\beta_{0}}
\bh^{\e 0}(K,M)^{\wedge}\\
&\Sha^{\e 1}\be(M).&& }
$$
The above is an exact sequence of profinite groups and continuous
homomorphisms, so each morphism is strict [1, \S III.2.8, p.237].
Consequently, the dual of the preceding sequence is also exact [15,
Theorem 23.7, p.196]. Exchanging the roles of $M$ and $M^{*}$ in
this dual exact sequence and noting that $(\bh^{\e
0}(K,M^{*})^{\wedge})^{D}=(\bh^{\e 0}(K,M^{*})^{D})_{\text{tors}}$
and $({\bh}^{-1}(K,M^{*})^{\wedge})^{D}={\bh}^{-1}(K,M^{*})^{D}$
(since ${\bh}^{-1}(K,M^{*})$ is finitely generated by [7, Lemma 2.1,
p.98]), we obtain an exact sequence
\[
\xymatrix{&\Sha^{\e 1}(M)\ar@{^{(}->}[d]&&\\
&\bh^{\e 1}(K,M)\ar[r]&{\bp}^{\e 1}(M)_{{\rm{tors}}}\ar[r]&(\bh^{\e
0}(K,M^{*})^{D})_{\text{tors}}
\ar[d]\\
0&\ar[l]{\bh}^{-1}(K,M^{*})^{D}&\ar[l]\displaystyle
\bigoplus_{{\rm{all}}\,v}{\bh}^{\e 2}(K_{v},M)&\ar[l]\bh^{\e
2}(K,M).
 }
\]
The sequence of the theorem may now be obtained by splicing together
the preceding two exact sequences.
\end{proof}

\section{The generalized Cassels-Tate dual exact sequence}

For $i=1$ or 2, define
$$
\Sha^{\e i}(T(M))=\krn\!\be\left[\e H^{\e i}(K,
T(M))\ra\displaystyle\prod_{\text{all $v$}}H^{\e i}(K_{v},
T(M))\right]
$$
and
$$
\Sha^{\e i}(M)=\krn\!\be\left[\e \bh^{\e
i}(K,M)\ra\displaystyle\prod_{\text{all $v$}}\bh^{\e
i}(K_{v},M)\right],
$$
where the $v$-component of each of the maps involved is induced by
the natural morphism $\spec K_{v}\ra\spec K$.

\begin{proposition} There exists a perfect pairing
$$
\Sha^{\e 1}(T(M^{*}))\times\Sha^{\e 2}(\lbe M\lbe)\ra\bq/\bz,
$$
where the first group is profinite and the second is discrete and
torsion.
\end{proposition}
\begin{proof} The proof is similar to the proof of Proposition 3.3.
\end{proof}

Let $S$ be any finite set of primes of $K$ and define, for $i=1$ or
$2$,
$$
\Sha^{\e i}_{S}(T(M))=\krn\!\be\left[\e H^{\e i}(K,
T(M))\ra\displaystyle\prod_{v\notin S}H^{\e i}(K_{v},T(M))\right]
$$
and
$$
\Sha^{\e i}_{S}(M)=\krn\!\left[\e {\bh}^{\e i}(K,
M)\ra\displaystyle\prod_{v\notin S}\bh^{\e i}(K_{v},M)\right].
$$
Thus $\Sha^{\e i}_{\e\emptyset}(T(M))=\Sha^{\e i}(T(M))$ and
$\Sha^{\e i}_{\e\emptyset}(M)=\Sha^{\e i}(M)$. Now partially order
the family of finite sets $S$ by defining $S\leq S^{\e\prime}$ if
$S\subset S^{\e\prime}$. Then $\Sha^{\e 1}_{S}(\lbe
M\lbe)\subset\Sha^{\e 1}_{S^{\prime}}(\lbe M\lbe)$ for $S\leq
S^{\e\prime}$. Set
$$
\Sha^{\e 1}_{\e\omega}(\lbe
M\lbe)=\displaystyle\varinjlim_{S}\Sha^{\e 1}_{S}(\lbe
M\lbe)=\displaystyle\bigcup_{S}\Sha^{\e 1}_{S}(\lbe
M\lbe)\,\subset\, \bh^{\e 1}(K,M),
$$
where the transition maps in the direct limit are the inclusion
maps. Thus $\!\!\Sha^{\e 1}_{\e\omega}(\lbe M\lbe)$ is the subgroup
of $\bh^{\e 1}(K,M)$ of all classes which are locally trivial at all
but finitely many places of $K$. Clearly, for each $S$ as above,
there exists an exact sequence of discrete torsion groups
$$
0\ra \Sha^{\e 1}(\lbe M^{*}\lbe)\ra\Sha^{\e 1}_{S}(\lbe
M^{*}\lbe)\ra\prod_{v\in S}\bh^{\e 1}(K_{v},M^{*})
$$
whose dual is an exact sequence of profinite groups
\begin{equation}
\prod_{v\in S}\bh^{\e
0}(K_{v},M)^{\wedge}\overset{\widehat{\theta}_{\be_S}}\longrightarrow
\Sha^{\e 1}_{S}(\lbe M^{*}\lbe)^{D}\ra\Sha^{\e 1}(\lbe
M^{*}\lbe)^{D}\ra 0.
\end{equation}
The map $\widehat{\theta}_{\lbe S}$ is given by
$$
\widehat{\theta}_{S}((m_{v}))(\xi)=\sum_{v\in
S}\,(m_{v},\xi\be\!\mid_{K_{v}})_{v}
$$
where, for each $v\in S$, $(-,-)_{v}$ is the pairing of [7, Theorem
2.3(2) or Proposition 2.9] and $\xi\e\vert_{K_{v}}$ is the image of
$\,\xi\in\!\!\Sha^{\e 1}_{S}(M^{*})\subset\bh^{\e 1}(K,M^{*})$ in
$\bh^{\e 1}(K_{v},M^{*})$ under the map induced by $\spec
K_{v}\ra\spec K$. We define
$\widehat{\theta}=\varprojlim_{S}\widehat{\theta}_{\lbe
S}\colon\prod_{\,\text{all }v}\bh^{\e
0}(K_{v},M)^{\wedge}\ra\Sha^{\e 1}_{\omega}(\lbe M^{*}\lbe)^{D}$.

\begin{proposition} There exists a canonical exact sequence
$$
T{\rm{Sel}}(M)^{\wedge}\overset{\widehat{\phi}}\longrightarrow
\displaystyle \prod_{{\rm{all}}\, v} \bh^{\e
0}(K_{v},M)^{\wedge}\overset{\widehat{\theta}}\longrightarrow\Sha^{\e
1}_{\omega}(M^{*})^{D}.
$$
Further, the map $\widehat{\phi}$ factors as
$$
T{\rm{Sel}}(M)^{\wedge}\overset{\beta_{0}}\longrightarrow\bp^{\e
0}(M)^{\wedge}\ra\displaystyle \prod_{{\rm{all}}\, v} \bh^{\e
0}(K_{v},M)^{\wedge},
$$
where the second map is the canonical one.
\end{proposition}
\begin{proof} The sequence of Proposition 3.10 is an exact sequence
of profinite groups and strict morphisms, so its dual
$$
\bh^{\e 1}(K,M^{*})\ra{\bp}^{\e
1}(M^{*})_{\e{\rm{tors}}}\overset{\beta_{0}^{D}}\longrightarrow
(T{\rm{Sel}}(M)^{\wedge})^{D}
$$
is also exact (cf. proof of Theorem 3.11). The above sequence
induces an exact sequence of discrete groups
$$
\Sha^{\e 1}_{S}(M^{*})\ra\prod_{v\in S}\bh^{\e 1}(K_{v},M^{*})\ra
(T{\rm{Sel}}(M)^{\wedge})^{D}
$$
whose dual is an exact sequence
$$
T{\rm{Sel}}(M)^{\wedge}\ra \displaystyle \prod_{v\in S} \bh^{\e
0}(K_{v},M)^{\wedge}\overset{\widehat{\theta}_{S}}\longrightarrow
\Sha^{\e 1}_{S}(M^{*})^{D}.
$$
Taking the inverse limit over $S$ above and noting that the inverse
limit functor is exact on the category of profinite groups [14,
Proposition 2.2.4, p.32], we obtain the exact sequence of the
proposition. That $\widehat{\phi}$ has the stated factorization
follows from the proof.
\end{proof}

\begin{proposition} There exists a canonical isomorphism
$$
\krn\!\!\left[T{\rm{Sel}}(M)^{\wedge}\overset{\widehat{\phi}}
\longrightarrow\displaystyle\prod_{\,{\rm{all}}\,v}\bh^{\e
0}(K_{v},M)^{\wedge}\right]=\Sha^{\e 2}(M^{*})^{D},
$$
where $\widehat{\phi}$ is the map of Proposition 4.2.
\end{proposition}
\begin{proof} By Proposition 4.2 and the fact that
$\krn\beta_{0}=\krn\theta_{0}=\!\!\Sha^{\e 2}(M^{*})^{D}$ by
Propositions 3.3 and 3.7, it suffices to check that the canonical
map ${\bp}^{\e 0}\!\left(M\be\right)^{\wedge}\ra \prod_{\,\text{all
}v}\bh^{\e 0}(K_{v},M)^{\wedge}$ is injective. The argument is
similar to that used in the proof of Lemma 3.6. Let $x\in{\bp}^{\e
0}\!\left(M\be\right)^{\wedge}$ be nonzero. There exists an open
subgroup $U\subset{\bp}^{\e 0}\!\left(M\be\right)$ of finite index
$n$ (say) such that the $U$-component of $x$, $x_{U}+U\in{\bp}^{\e
0}\!\left(M\be\right)/U$ is nonzero, i.e., $x_{U}\notin U$. Then
$x_{U}\notin n{\bp}^{\e 0}\!\left(M\be\right)$, whence
$(x_{U})_{v}\notin n\bh^{\e 0}(K_{v},M)$ for some $v$. Thus the
image of $x$ in $\bh^{\e 0}(K_{v},M)/n=(\bh^{\e
0}(K_{v},M)/n)^{\wedge}$ is nonzero. Since the map ${\bp}^{\e
0}\!\left(M\be\right)^{\wedge}\ra(\bh^{\e 0}(K_{v},M)/n)^{\wedge}$
factors through $\bh^{\e 0}(K_{v},M)^{\wedge}$, the image of $x$ in
$\bh^{\e 0}(K_{v},M)^{\wedge}$ is nonzero.
\end{proof}

As noted earlier, the inverse limit functor is exact on the category
of profinite groups, so the inverse limit over $S$ of (8) is an
exact sequence
$$
\displaystyle\prod_{\text{all }v}\bh^{\e
0}(K_{v},M)^{\wedge}\overset{\widehat{\theta}}\longrightarrow
\Sha^{\e 1}_{\omega}(\lbe M^{*}\lbe)^{D}\ra\Sha^{\e 1}(\lbe
M^{*}\lbe)^{D}\ra 0.
$$
We now use Propositions 4.2 and 4.3 to extend the above exact
sequence to the left and obtain

\begin{theorem} {\rm{(The generalized Cassels-Tate dual exact
sequence)}} There exists a canonical exact sequence of profinite
groups
$$\begin{array}{rcl}
0\ra\Sha^{\e 2}(M^{*})^{D}&\ra & T\e{\rm{Sel}}(M)^{\wedge}\ra
\displaystyle\prod_{{\rm{all}}\,v}\bh^{\e
0}(K_{v},M)^{\wedge}\\
&\ra &\Sha^{\e 1}_{\omega}(M^{*})^{D}\ra \Sha^{\e 1}(M^{*})^{D} \ra
0.\qed
\end{array}
$$
\end{theorem}

\begin{corollary} There exists a canonical exact sequence of
discrete torsion groups
$$\begin{array}{rcl}
0\ra\Sha^{\e 1}(\e M)&\ra & \Sha^{\e 1}_{\omega}(M)\ra
\displaystyle\bigoplus_{{\rm{all}}\,v}\bh^{\e
1}(K_{v},M)\\
&\ra &(T\e{\rm{Sel}}(M^{*})^{\wedge})^{D}\ra\Sha^{\e 2}(M)\ra 0.\qed
\end{array}
$$
\end{corollary}

We conclude this paper with the following result, which extends [8,
Theorem 1.2] to the function field case.
\begin{theorem} Let $K$ be a global function field and let $M$
be a 1-motive over $K$. Assume that $\!\!\Sha^{\e 1}(M)$ is finite.
Then there exists an exact sequence
$$
0\ra\overline{\bh^{0}(K,M)}\ra
\displaystyle\prod_{{\rm{all}}\,v}\bh^{\e 0}(K_{v},M)\ra\Sha^{\e
1}_{\omega}(M^{*})^{D}\ra \Sha^{\e 1}(M)\ra 0,
$$
where $\overline{\bh^{0}(K,M)}$ denotes the closure of the diagonal
image of $\,\bh^{0}(K,M)$ in $\prod_{\,{\rm{all}}\,v}\bh^{\e
0}(K_{v},M)$.
\end{theorem}
\begin{proof} The proof is essentially the same as that of
[8, Theorem 1.2], noting that
$T\e{\rm{Sel}}(M)^{\wedge}=\bh^{0}(K,M)^{\wedge}$ if $\Sha^{\e
1}(M)$ is finite and using Proposition 4.2 in place of [8,
Proposition 5.3(1)].
\end{proof}

\end{document}